\documentclass[12pt,reqno]{amsart}
\usepackage{times}

\newtheorem{thm}{Theorem}
\newtheorem{cor}[thm]{Corollary}
\newtheorem{lem}[thm]{Lemma}

\newcommand{\C}{{\mathbb C}}
\newcommand{\D}{{\mathbb D}}
\newcommand{\inc}{\int_\C}

\newcommand{\Z}{{\mathbb Z}}

\begin{document}

\title{Maximal Zero Sequences for Fock Spaces}

\author{Kehe Zhu}
\address{Department of Mathematics and Statistics, State University of New York,
Albany, NY 12222, USA.}
\email{kzhu@math.albany.edu}

\begin{abstract}
A sequence $Z$ in the complex plane $\C$ is called a zero sequence for the Fock space 
$F^p_\alpha$ if there exists a function $f\in F^p_\alpha$, not identically zero, such 
that $Z$ is the zero set of $f$, counting multiplicities. We show that there exist zero 
sequences $Z$ for $F^p_\alpha$ with the following properties: (1) For any $a\in\C$ the 
sequence $Z\cup\{a\}$ is no longer a zero sequence for $F^p_\alpha$; (2) the space 
$I_Z$ consisting of all functions in $F^p_\alpha$ that vanish on $Z$ is one dimensional. 
These $Z$ are naturally called maximal zero sequences for $F^p_\alpha$.
\end{abstract}

\maketitle

\section{Introduction}

Let $\Omega$ be a domain in the complex plane $\C$ and let $X$ be a space of analytic
functions on $\Omega$. A sequence $Z$ in $\Omega$ is called a zero sequence (or zero set)
for $X$ if there exists a function $f\in X$ such that $f$ vanishes exactly on $Z$, 
counting multiplicities.

A classical example is the Hardy space $H^p$ of the unit disk $\D$. In this case,
$Z=\{z_n\}$ is a zero sequence for $H^p$ if and only if
$$\sum_n(1-|z_n|)<\infty,$$
which is called the Blaschke condition. Other examples that have been extensively
studied in complex analysis include Bergman spaces, the Dirichlet space, and the disk algebra.
See \cite{D,G,HKZ}.

In all the examples mentioned above, if $Z$ is a zero sequence for $X$, then 
\begin{enumerate}
\item[(i)] $Z\cup\{a_1,\cdots,a_k\}$ remains a zero sequence for $X$, where 
$a_1,\cdots,a_k$ are arbitrary additional points from the underlying domain.
\item[(ii)] The space $I_Z$ of functions in $X$ that vanish on $Z$ (not necessarily exactly 
on $Z$) is always infinite dimensional.
\end{enumerate}
In fact, it is easy to see that both (i) and (ii) hold whenever the space $X$ is invariant 
under multiplication by polynomials. More generally, properties (i) and (ii) hold whenever 
$X$ satisfies the following condition: for any point $a$ in the underlying domain, there 
exists a function $f_a$ such that $f_a$ vanishes exactly at $a$ and $f_a$ is a pointwise 
multiplier of $X$. The underlying domain does not have to be bounded and polynomials do 
not need to be pointwise multipliers. In particular, Hardy spaces of the upper half-plane 
satisfy this condition.

The purpose of this note is to examine properties (i) and (ii) above in the case of Fock 
spaces. Thus for any $0<\alpha<\infty$ and $0<p<\infty$ we consider the Fock space 
$F^p_\alpha$, consisting of entire functions $f$ such that
$$\|f\|^p_{p,\alpha}=\frac\alpha\pi\inc\left|f(z)e^{-\frac\alpha2|z|^2}\right|^p
\,dA(z)<\infty.$$
When $p=\infty$, we define $F^\infty_\alpha$ to be the space of entire functions $f$ such that
$$\|f\|_{\infty,\alpha}=\sup_{z\in\C}|f(z)|e^{-\frac\alpha2|z|^2}<\infty.$$
See \cite{JPR,T,Z2} for more information about Fock spaces.

We will show that there exist zero sequences for $F^p_\alpha$ such that neither (i) nor (ii) 
holds. In fact, we will produce examples of zero sequences $Z$ for $F^p_\alpha$ such that
\begin{enumerate}
\item[(i)] For any $a\in\C$ the sequence $Z\cup\{a\}$ is no longer a zero sequence 
for $F^p_\alpha$.
\item[(ii)] $\dim(I_Z)=1$.
\end{enumerate}
These are clearly very extreme cases. Such $Z$ will be called maximal zero sequences
for $F^p_\alpha$, because it cannot be expanded to another zero sequence for $F^p_\alpha$.
See \cite{Z2} for other pathological properties of zero sequences for Fock spaces.

Note that $Z'=Z\cup\{a_1,\cdots,a_k\}$ should not be looked at as a purely set-theoretic 
notation. It should be understood in the context of zero sequences for analytic functions
and multilicities of the zeros are part of the notation as well. For example, if
$Z'=\{1,1,2\}\cup\{2,2,4\}$, then a function $f$ vanishes on $Z'$ when $f$ has a zero of
order $2$ at $z=1$, a zero of order $3$ at $z=2$, and a simple zero at $z=4$. If we do 
not say that $f$ vanishes {\it exactly} on $Z'$, then additional zeros are permitted.

\section{Weierstrass $\sigma$-functions}

Our examples will be based on square lattices in the complex plane and the associated
Weierstrass $\sigma$-functions. More specifically, for any $0<\alpha<\infty$ we consider
the square lattice
$$\Lambda_\alpha=\left\{\omega_{mn}=\sqrt{\frac\pi\alpha}\,(m+in):(m,n)\in\Z^2\right\},$$
where $\Z=\{0,\pm1,\pm2,\pm3,\cdots\}$ denotes the set of all integers. The Weierstrass
$\sigma$-function associated to $\Lambda_\alpha$ is defined by
$$\sigma_\alpha(z)=z\prod_{(m,n)\not=(0,0)}\left[\left(1-\frac z{\omega_{mn}}\right)
\exp\left(\frac z{\omega_{mn}}+\frac{z^2}{2\omega_{mn}^2}\right)\right].$$
It is well known that $\sigma_\alpha$ is an entire function and its zeros are exactly
the points in the lattice $\Lambda_\alpha$.

The Weierstrass $\sigma$ functions play an essential role in the study of Fock spaces.
For example, these functions are used in \cite{S,SW} to characterize interpolating and
sampling sequences for Fock spaces. There are two properties of the Weierstrass 
$\sigma$-functions that will be critical to us here. We state them as the following two 
lemmas.

\begin{lem}
For any $\alpha>0$ the function 
$$f(z)=|\sigma_\alpha(z)|e^{-\frac\alpha2|z|^2}$$
is doubly periodic with periods $\sqrt{\pi/\alpha}$ and $\sqrt{\pi/\alpha}\,i$.
\label{1}
\end{lem}

\begin{proof}
See \cite{WW}.
\end{proof}

As a consequence of this result, we see that $\sigma_\alpha\in F^\infty_\alpha$, but
$\sigma_\alpha\not\in F^p_\alpha$ for any $0<p<\infty$. In fact, if we let $R_{mn}$
denote the square centered at $\omega_{mn}$ with horizontal and vertical side-lengths 
$\sqrt{\pi/\alpha}$, then by Lemma~\ref{1} and a change of variables,
\begin{eqnarray*}
\|\sigma_\alpha\|^p_{p,\alpha}&=&\frac\alpha\pi\inc\left|\sigma_\alpha(z)e^{-\frac\alpha2|z|^2}
\right|^p\,dA(z)\\
&=&\frac\alpha\pi\sum_{m,n}\int_{R_{mn}}\left|\sigma_\alpha(z)e^{-\frac\alpha2|z|^2}\right|^p
\,dA(z)\\
&=&\frac\alpha\pi\sum_{m,n}\int_{R_{00}}\left|\sigma_\alpha(z)e^{-\frac\alpha2|z|^2}
\right|^p\,dA(z)\\
&=&\infty.
\end{eqnarray*}

It is well known that every function $f\in F^p_\alpha$ satisfies the following pointwise 
estimate
\begin{equation}
|f(z)|\le\|f\|_{p,\alpha}e^{\frac\alpha2|z|^2},\qquad z\in\C.
\label{eq1}
\end{equation}
See \cite{JPR,T,Z2}. It follows from this and the definition of Fock spaces that 
$F^p_\alpha\subset F^q_\beta$ whenever $0<\alpha<\beta<\infty$, where $0<p\le\infty$ 
and $0<q\le\infty$. Therefore, for $0<\alpha_1<\alpha<\alpha_2<\infty$ and $0<p\le\infty$, 
we always have $\sigma_\alpha\in F^p_{\alpha_2}$ but $\sigma_\alpha\not\in F^p_{\alpha_1}$.

\begin{lem}
If $0<p<\infty$ and $f$ is a function in $F^p_\alpha$ that vanishes on $\Lambda_\alpha$,
then $f$ is identically zero.
\label{2}
\end{lem}

\begin{proof}
See \cite{SW}.
\end{proof}

Consequently, the square lattice $\Lambda_\alpha$ is a zero sequence for $F^\infty_\alpha$
but not a zero sequence for $F^p_\alpha$ when $0<p<\infty$. Another consequence is the 
following.

\begin{cor}
Suppose $0<\alpha<\beta<\infty$, $0<p\le\infty$, and $0<q\le\infty$. Then every zero sequence
for $F^p_\alpha$ is a zero sequence for $F^q_\beta$, but there exists a zero sequence
for $F^q_\beta$ that is not a zero sequence for $F^p_\alpha$.
\label{3}
\end{cor}

\begin{proof}
The first assertion follows from the embedding $F^p_\alpha\subset F^q_\beta$. To prove
the second assertion, pick some $\gamma\in(\alpha,\beta)$. By Lemma~\ref{1}, the function 
$\sigma_\gamma$ belongs to $F^q_\beta$, so that $\Lambda_\gamma$ is a zero sequence for
$F^q_\beta$. If $f$ is a function in $F^p_\alpha\subset F^2_\gamma$ that vanishes on 
$\Lambda_\gamma$, then by Lemma~\ref{2}, $f$ must be identically zero, so $\Lambda_\gamma$
is not a zero sequence for $F^p_\alpha$.
\end{proof}

The more interesting problem for us is when $\alpha=\beta$: do $F^p_\alpha$ and 
$F^q_\alpha$ have different zero sequences whenever $p\not=q$? Although we do not 
have a complete answer, it is easy to exhibit particular examples of such pairs that 
do not have the same zero sequences.

The simplest example is $Z=\Lambda_\alpha$, which is a zero sequence for $F^\infty_\alpha$, 
but not a zero sequence for any $F^p_\alpha$ when $0<p<\infty$. Similarly, $Z=\Lambda_\alpha-\{0\}$ is a zero sequence for $F^p_\alpha$ when $p>2$,
because the function $f(z)=\sigma_\alpha(z)/z$ belongs to $F^p_\alpha$ if and only if
$p>2$ (see Lemma~\ref{1} and calculations below). However, this $Z$ is not a zero sequence 
for $F^2_\alpha$. To see this, suppose $f$ is a function in $F^2_\alpha$ such that $f$ 
vanishes on $Z$. By Weierstrass factorization, we have $f(z)=[\sigma_\alpha(z)/z]g(z)$ 
for some entire function $g$. Let $\Omega=\C-R_{00}$ and
$$I(f)=\int_{\Omega}\left|f(z)e^{-\frac\alpha2|z|^2}\right|^2\,dA(z).$$
Then by the double periodicity of the function $\sigma_\alpha(z)e^{-\frac\alpha2|z|^2}$,
\begin{eqnarray*}
I(f)&=&\sum_{(m,n)\not=(0,0)}\int_{R_{mn}}\left|\sigma_\alpha(z)e^{-\frac\alpha2|z|^2}
\right|^2\left|\frac{g(z)}{z}\right|^2\,dA(z)\\
&=&\sum_{(m,n)\not=(0,0)}\int_{R_{00}}\left|\sigma_\alpha(z)e^{-\frac\alpha2|z|^2}\right|^2
\left|\frac{g(z+\omega_{mn})}{z+\omega_{mn}}\right|^2\,dA(z).
\end{eqnarray*}
Let $D$ denote the disk $|z|<\sqrt{\pi/\alpha}/100$ and choose a positive constant $C_1$
such that 
$$|\sigma_\alpha(z)|e^{-\frac\alpha2|z|^2}>C_1,\qquad z\in R_{00}-D.$$
We then have
$$I(f)\ge C_1\sum_{(m,n)\not=(0,0)}\int_{R_{00}-D}\left|
\frac{g(z+\omega_{mn})}{z+\omega_{mn}}\right|^2\,dA(z).$$
When $(m,n)\not=(0,0)$, the function $h(z)=g(z+\omega_{mn})/(z+\omega_{mn})$ is analytic
on $R_{00}$. It follows easily from the subharmonicity of $|h(z)|^2$ that there is 
another positive constant $C_2$, independent of $(m,n)$, such that
$$I(f)\ge C_2\sum_{(m,n)\not=(0,0)}\int_{R_{00}}\left|\frac{g(z+\omega_{mn})}
{z+\omega_{mn}}\right|^2\,dA(z)=C_2\int_{\Omega}\left|\frac{g(z)}{z}\right|^2\,dA(z).$$
It is easy (using polar coordinates, for example) to show that
$$\int_{\Omega}\left|\frac{g(z)}z\right|^2\,dA(z)<\infty$$
if and only if $g$ is identically zero (here the exponent $2$ is critical). 
Therefore, $f\in F^2_\alpha$ implies that $f$ is identically zero. In other words, 
the sequence $Z$ cannot possibly be a zero set for $F^2_\alpha$.

The above argument actually shows that $Z=\Lambda_\alpha-\{0\}$ is a uniqueness set 
for $F^2_\alpha$. Recall that a set $Z$ in $\C$ is called a uniqueness set (or set of
uniqueness) for $F^p_\alpha$ if every function in $F^p_\alpha$ that vanishes on $Z$
must be identically zero. It is also easy to see that the arguments above still work if 
the point $0$ is replaced by any other point in $\Lambda_\alpha$.

On the other hand, if $Z$ is the resulting sequence when two points $a$ and $b$ are 
removed from $\Lambda_\alpha$, then the function 
$$f(z)=\frac{\sigma_\alpha(z)}{(z-a)(z-b)}$$
belongs to $F^2_\alpha$ and has $Z$ as its zero sequence. Therefore, $Z$ is a zero set for
$F^2_\alpha$. Consequently, it is possible to go from a uniqueness set to a zero set by
removing just one point. Equivalently, it is possible to add just a single point to a zero set
of $F^2_\alpha$ so that the resulting sequence becomes a uniqueness set for $F^2_\alpha$.
This shows how unstable the zero sets for Fock spaces are! See \cite{P1,P2} for applications
of this rigidity in quantumn physics.

We also observe that for any positive integer $N$ with $Np>2$, if $Z$ is an $F^q_\alpha$ 
zero set, where $0<p<q\le\infty$, and if $N$ points are removed from $Z$, then the 
remaining sequence becomes an $F^p_\alpha$ zero set. In fact, if $Z$ is the zero sequence 
of a function $f\in F^q_\alpha$, not identically zero, and $Z'=Z-\{z_1,\cdots,z_N\}$, 
then $Z'$ is the zero sequence of the function
$$g(z)=\frac{f(z)}{(z-z_1)\cdots(z-z_N)},$$
which is easily seen to be in $F^p_\alpha$. In fact, if $R>\max(|z_1|,\cdots,|z_N|)$, then it 
follows from the pointwise estimate (\ref{eq1}) that there exists a positive constant $C$ such that
$$\int_{|z|>R}|g(z)e^{-\frac\alpha2|z|^2}|^p\,dA(z)\le C\int_{|z|>R}\frac{dA(z)}{|z-z_1|^p
\cdots|z-z_N|^p}<\infty.$$
Therefore, zero sets for $F^p_\alpha$ and $F^q_\alpha$ may be different, but they are
not that much different! In other words, the difference is only in a finite number of points.

\section{Maximal zero sets for $F^p_\alpha$}

Let $Z$ be a zero sequence for $F^p_\alpha$ and let $I_Z$ denote the set of functions $f$
in $F^p_\alpha$ such that $f$ vanishes on $Z$ (but not necessarily exactly on $Z$). In 
the classical theories of Hardy and Bergman spaces, the space $I_Z$ is always infinite 
dimensional. This is no longer true for Fock spaces.

\begin{thm}
Let $k$ be any positive integer or $\infty$. Then there exists a zero sequence $Z$ for 
$F^p_\alpha$ such that $\dim(I_Z)=k$.
\label{4}
\end{thm}

\begin{proof}
The case $k=\infty$ is trivial; any finite sequence $Z$ will work. So we assume that $k$ is
a positive integer in the rest of the proof.

We first consider the case $p=\infty$ and $k>1$. In this case, we consider
$Z=\Lambda_\alpha-\{a_1,\cdots,a_{k-1}\}$, where $a_1,\cdots,a_{k-1}$ are (any) distinct 
points in $\Lambda_\alpha$, and
$$f(z)=\frac{\sigma_\alpha(z)}{(z-a_1)\cdots(z-a_{k-1})}.$$
It follows from Lemma \ref{1} that $f\in F^\infty_\alpha$ and $Z$ is exactly the zero sequence
of $f$. Furthermore, if $h$ is a polynomial of degree less than or equal to $k-1$, then the
function $f(z)h(z)$ is still in $F^\infty_\alpha$. 

On the other hand, if $F$ is any function in 
$F^\infty_\alpha$ that vanishes on $Z$, then we can write 
$$F(z)=f(z)g(z)=\frac{\sigma_\alpha(z)g(z)}{(z-a_1)\cdots(z-a_{k-1})},$$
where $g$ is an entire function. For any positive integer $n$ let $C_n$ be the boundary 
of the square centered at $0$ with horizontal and vertical side-length
$(2n+1)/\sqrt{\pi/\alpha}$. It is clear that 
$$d(C_n,\Lambda_\alpha)\ge\sqrt{\pi/\alpha}/2,\qquad n\ge1.$$
So there exists a positive constant $C$ such that
$$|\sigma_\alpha(z)|e^{-\frac\alpha2|z|^2}\ge C,\qquad z\in C_n,n\ge1.$$
This together with the assumption that $F\in F^\infty_\alpha$ implies that there exists
another positive constant $C$ such that
\begin{equation}
|g(z)|\le C|z-a_1|\cdots|z-a_{k-1}|
\label{eq2}
\end{equation}
for all $z\in C_n$ and $n\ge1$. By Cauchy's integral estimates, the function $g$ must be
a polynomial of degree at most $k-1$.

Therefore, when $p=\infty$, $k>1$, and $Z=\Lambda_\alpha-\{a_1,\cdots,a_{k-1}\}$,
we have shown that a function $F\in F^\infty_\alpha$ vanishes on $Z$ if and only if
$$F(z)=\frac{\sigma_\alpha(z)h(z)}{(z-a_1)\cdots(z-a_{k-1})},$$
where $h$ is a polynomial of degree less than or equal to $k-1$. This shows that
$\dim(I_Z)=k$. 

When $p=\infty$ and $k=1$, we simply take $Z=\Lambda_\alpha$. The arguments
above can be simplified to show that a function $F\in F^\infty_\alpha$ vanishes 
on $Z$ if and only if $F=c\sigma_\alpha$ for some constant $c$.

Next, we assume that $0<p<\infty$ and $k$ is a positive integer. In this case, we let 
$N$ denote the smallest positive integer such that $Np>2$, or equivalently,
\begin{equation}
\int_{|z|>1}\left|\frac{\sigma_\alpha(z)e^{-\frac\alpha2|z|^2}}{z^N}\right|^p
\,dA(z)<\infty.
\label{eq3}
\end{equation}
Remove any $N+k-1$ points $\{a_1,\cdots,a_{N+k-1}\}$ from $\Lambda_\alpha$ and denote 
the remaining sequence by $Z$. Then $Z$ is the zero sequence of the function
$$\frac{\sigma_\alpha(z)}{(z-a_1)\cdots(z-a_{N+k-1})},$$
which belongs to $F^p_\alpha$ in view of (\ref{eq3}). In fact, if $g$ is any polynomial 
of degree less than or equal to $k-1$, then it follows from (\ref{eq3}) that $g$ times 
the above function belongs to $I_Z$. 

Conversely, if $f$ is any function in $F^p_\alpha$ that vanishes on $Z$, then we can write
$$f(z)=\frac{\sigma_\alpha(z)g(z)}{(z-a_1)\cdots(z-a_{N+k-1})},$$
where $g$ is an entire function. Since $F^p_\alpha\subset F^\infty_\alpha$, it follows
from (\ref{eq2}) and Cauchy's integral estimates that that $g$ is a polynomial with 
degree less than or equal to $N+k-1$. If the degree of $g$ is $j>k-1$, then
$$\frac{g(z)}{(z-a_1)\cdots(z-a_{N+k-1})}\sim\frac1{z^{N+k-1-j}},\qquad z\to\infty.$$
This together with $f\in F^p_\alpha$ shows that (\ref{eq3}) still holds when $N$ is
replaced by $N+k-1-j$, which contradicts our minimality assumption on $N$. Thus 
$j\le k-1$, which shows that $I_Z$ is $k$-dimensional.
\end{proof}

\begin{lem}
Let $Z$ be a zero sequence for $F^p_\alpha$ and $\dim(I_Z)=k<\infty$. Then
$Z\cup\{a_1,\cdots,a_k\}$ is always a uniqueness set for $F^p_\alpha$.
\label{5}
\end{lem}

\begin{proof}
Let $Z'=Z\cup\{a_1,\cdots,a_k\}$. If there exists a function $f\in F^p_\alpha$, not
identically zero, such that $f$ vanishes on $Z'$. Then the functions
$$f(z),\quad \frac{f(z)}{z-a_1},\quad \cdots,\quad \frac{f(z)}{z-a_k},$$
all belong to $F^p_\alpha$ and vanish on $Z$ (obvious adjustments should be made when
there are zeros of high multiplicity). They are clearly linearly independent,
so the dimension of $I_Z$ is at least $k+1$. Therefore, the condition $\dim(I_Z)\le k$
implies that $Z\cup\{a_1,\cdots,a_k\}$ is always a uniqueness set for $F^p_\alpha$.
\end{proof}

\begin{lem}
Let $Z$ be a zero sequence for $F^p_\alpha$ and $\dim(I_Z)>k$. Then 
$Z\cup\{a_1,\cdots,a_k\}$ is never a uniqueness set for $F^p_\alpha$.
\label{6}
\end{lem}

\begin{proof}
If $\dim(I_Z)>k$, there exist $k+1$ linearly independent functions in $I_k$, say $f_1,\cdots,f_k,f_{k+1}$. Fix any collection $\{a_1,\cdots,a_k\}$ and 
let $Z'=Z\cup\{a_1,\cdots,a_k\}$. Consider the linear combination
\begin{equation}
f=c_1f_1+\cdots+c_{k+1}f_{k+1},
\label{eq4}
\end{equation}
and the system of linear equations
$$c_1f(a_j)+\cdots+c_{k+1}f_{k+1}(a_j)=0,\qquad 1\le j\le k.$$
Again, obvious adjustments should be made when there are zeros of high multiplicity.
This homogeneous system has $k$ equations but $k+1$ variables, so it always has nonzero 
solutions $c_j$, $1\le j\le k$. With such a choice of $c_j$, the function $f$ defined
in (\ref{eq4}) is not identically zero but vanishes on $Z'$. So $Z'$ is not a 
uniqueness set for $F^p_\alpha$.
\end{proof}

\begin{cor}
Suppose $Z$ is a zero sequence for $F^p_\alpha$ and $k$ is a positive integer. Then the
following conditions are equivalent.
\begin{enumerate}
\item[(a)] $\dim(I_Z)\le k$.
\item[(b)] $Z\cup\{a_1,\cdots,a_k\}$ is a uniqueness set for $F^p_\alpha$ for all
$\{a_1,\cdots,a_k\}$.
\item[(c)] $Z\cup\{a_1,\cdots,a_k\}$ is a uniqueness set for $F^p_\alpha$ for some
$\{a_1,\cdots,a_k\}$.
\end{enumerate}
\label{7}
\end{cor}

\begin{proof}
That (a) implies (b) follows from Lemma~\ref{5} and the fact that expanding a uniqueness
set always results in a uniqueness set. It is trivial that (b) implies (c). Lemma~\ref{6}
shows that (c) implies (a).
\end{proof}

\begin{cor}
Let $Z$ be a zero sequence for $F^p_\alpha$ and $k$ be a positive integer. Then the 
following conditions are equivalent.
\begin{enumerate}
\item[(a)] $\dim(I_Z)=k$.
\item[(b)] For any $\{a_1,\cdots,a_k\}$ the sequence $Z\cup\{a_1,\cdots,a_{k-1}\}$ is not
a uniqueness set for $F^p_\alpha$ but the sequence $Z\cup\{a_1,\cdots,a_k\}$ is.
\item[(c)] For some $\{a_1,\cdots,a_{k-1}\}$ the sequence $Z\cup\{a_1,\cdots,a_{k-1}\}$ 
is not a uniqueness set for $F^p_\alpha$ but $Z\cup\{b_1,\cdots,b_k\}$ is a uniqueness 
set for some $\{b_1,\cdots,b_k\}$.
\end{enumerate}
\label{8}
\end{cor}

\begin{proof}
This is a direct consequence of Corollary~\ref{7}.
\end{proof}

Therefore, if $\dim(I_Z)=k$ is a positive integer, then adding $k$ points to $Z$ always
results in a uniqueness set, but adding less than $k$ points never results in a uniqueness
set. We now show that the second assertion can be improved.

\begin{cor}
Suppose $Z$ is a zero sequence for $F^p_\alpha$, $\dim(I_Z)=k$ is a positive integer, 
and $j<k$. Then $Z\cup\{a_1,\cdots,a_j\}$ is always a zero sequence for $F^p_\alpha$.
\label{9}
\end{cor}

\begin{proof}
By Corollary~\ref{8}, $Z'=Z\cup\{a_1,\cdots,a_j\}$ is not a uniqueness set for $F^p_\alpha$. 
So there exists a function $f\in F^p_\alpha$, not identically zero, such that $f$ vanishes 
on $Z'$. By Corollary~\ref{8} again, the number of additional zeros of $f$ cannot 
exceed $k-j$. If these additional zeros $a$ are divided out of $f$ by the corresponding 
factors $z-a$, the resulting function is still in $F^p_\alpha$ and vanishes exactly on 
$Z'$. Thus $Z'$ is a zero sequence for $F^p_\alpha$.
\end{proof}

Consequently, if $Z$ is a zero sequence for $F^p_\alpha$ and $\dim(I_Z)=k$, then adding
less than $k$ points to $Z$ will always result in a zero sequence for $F^p_\alpha$ again,
but adding $k$ points to $Z$ will always result in a uniqueness set (which is certainly
not a zero sequence). Once again, this shows how unstable the zero sequences
of $F^p_\alpha$ are. The following result describes the structure of $I_Z$ when it is
finite dimensional.

\begin{thm}
Suppose $Z$ is a zero sequence for $F^p_\alpha$ and $\dim(I_Z)=k$ is a positive integer.
Then there exists a function $g\in I_Z$ such that $I_Z=gP_{k-1}$, where $P_{k-1}$ is
the set of all polynomials of degree less than or equal to $k-1$.
\label{10}
\end{thm}

\begin{proof}
Let $f$ be a function in $I_Z$, not identically zero. Then its zero sequence 
must be of the form $Z\cup\{a_1,\cdots,a_j\}$, where $j\le k-1$. Otherwise, we can 
come up with a set $Z'=Z\cup\{a_1,\cdots,a_k\}$ such that $f$ vanishes on $Z'$, which
is a contradition to Lemma~\ref{5}.

It follows from the previous paragraph that every nonzero function $f\in I_Z$ must 
have order $2$ and type $\alpha/2$. Otherwise, multiplication of $f$ by arbitrary 
polynomials will still produce functions in $I_Z$, and we get functions in $I_Z$ that 
have more than $k$ additonal zeros.

Now fix a function $g\in I_Z$ that vanishes exactly on $Z$. For any $f\in I_Z$ we
have the factorization $f=gPe^h$, where $P$ is a polynomial with $\deg(P)\le k-1$ and
$h$ is entire. Since both $f$ and $g$ are of order $2$ and type $\alpha/2$, the
function $h$ must be constant. This shows that $I_Z\subset gP_{k-1}$. But
$$\dim(I_Z)=k=\dim(gP_{k-1}),$$
so we must actually have $I_Z=gP_{k-1}$.
\end{proof}

Obviously, the function $g$ in Theorem~\ref{10} is essentially unique. More specifically,
any two such functions can only differ by a constant multiple. This essentially unique
function is determined by requiring it to have {\em exactly} $Z$ as its zero sequence.


\begin{thebibliography}{99}

\bibitem{D} P. Duren, {\sl Theory of $H^p$ Spaces}, Academic Press, New York, 1970.
\bibitem{G} J. Garnett, {\sl Bounded Analaytic Functions}, Academic Press, New York, 1981.
\bibitem{JPR} S. Janson, J. Peetre, and R. Rochberg, Hankel forms and the Fock
space, {\sl Revista Mat. Ibero-Amer.} {\bf 3} (1987), 61-138.
\bibitem{HKZ} H. Hedenmalm, B. Korenblum, and K. Zhu, {\sl Theory of Bergman
Spaces}, Springer-Verlag, New York, 2000.
\bibitem{P1} A. M. Perelomov, On the completeness of a system of coherent states,
{\sl Theor. Math. Phys.} {\bf 6} (1971), 156-164.
\bibitem{P2} A. M. Perelomov, {\sl Generalized Coherent States and Their Applications},
Springer-Verlag, Berlin, 1986.
\bibitem{S} K. Seip, Density theorems for sampling and interpolation in
the Bargmann-Fock space I, {\sl J. Reine Angew. Math.} {\bf 429} (1992), 91-106.
\bibitem{SW} K. Seip and R. Wallst\'en, Density theorems for sampling and
interpolation in the Bargmann-Fock space II, {\sl J. Reine Angew. Math.} {\bf 429}
(1992), 107-113.
\bibitem{T} J.Y. Tung, Fock spaces, Ph.D. thesis, University of Michigan, 2005.
\bibitem{WW} E. T. Whittaker and G. N. Watson, {\sl A Course of Modern Analysis}, 
4th edition, Cambridge Univ. Press, reprinted in 1996.
\bibitem{Z1} K. Zhu, Zeros of functions in Fock spaces, {\sl Complex Variables}
{\bf 21} (1993), 87-98.
\bibitem{Z2} K. Zhu, {\sl Analysis on Fock Spaces}, Springer-Verlag, New York, 2012.

\end{thebibliography}
\end{document}